\documentclass[12pt,final]{article}
\usepackage[margin=2.5cm,top=1.5cm]{geometry}
\usepackage[all,arc]{xy}
\usepackage{amsmath,amsthm,amssymb,enumerate}
\usepackage{color}
\newtheorem{con}{Conjecture}[section]
\newtheorem{thm}[con]{Theorem}

\newtheorem{cor}[con]{Corollary}
\newtheorem{lem}[con]{Lemma}

\theoremstyle{definition}

\newtheorem{rem}[con]{Remark}

\theoremstyle{remark}

\newcommand{\Romannum}[1]{\uppercase\expandafter{\romannumeral #1}}

\numberwithin{equation}{section}

\newcommand\keywordsname{Key words}
\newcommand\AMSname{AMS subject classifications}

\newenvironment{@abssec}[1]{%
     \if@twocolumn
       \section*{#1}%
     \else
       \vspace{.05in}\footnotesize
       \parindent .2in
         {\upshape\bfseries #1. }\ignorespaces
     \fi}
     {\if@twocolumn\else\par\vspace{.1in}\fi}

\begin{document}

\vskip6cm
\title{On a conjecture for the signless Laplacian eigenvalues
\footnote{Research supported by National Natural Science Foundation of China
(Nos.10901061, 11071088), the Zhujiang Technology New Star Foundation of
Guangzhou (No.2011J2200090), and Program on International Cooperation and Innovation, Department of Education, Guangdong Province 
(No.2012gjhz0007).}}
\author{Jieshan Yang \qquad Lihua You\footnote{{\it{Corresponding author:\;}}ylhua@scnu.edu.cn}}
\vskip.2cm
\date{{\small
School of Mathematical Sciences, South China Normal University,\\
Guangzhou, 510631, China\\
}} \maketitle

\begin{abstract}

 \vskip.3cm
 Let $G$ be a simple graph with $n$ vertices and $e(G)$ edges, and $q_1(G)\geq q_2(G)\geq\cdots\geq q_n(G)\geq0$ be the signless Laplacian eigenvalues of $G.$ Let $S_k^+(G)=\sum_{i=1}^{k}q_i(G),$ where $k=1, 2, \ldots, n.$ F. Ashraf et al. conjectured that $S_k^+(G)\leq e(G)+\binom{k+1}{2}$ for $k=1, 2, \ldots, n.$ In this paper, we give various upper bounds for $S_k^+(G),$ and prove that this conjecture is true for the following cases: connected graph with sufficiently large $k,$ unicyclic graphs and bicyclic graphs for all $k,$ and tricyclic graphs when $k\neq 3.$
\end{abstract}

\section{ Introduction}
\hskip.6cm
Let $G$ be a simple graph with vertex set $V(G)$ and edge set $E(G).$ The Laplacian matrix and the signless Laplacian matrix of $G$ are defined as $L(G)=D(G)-A(G)$ and $Q(G)=D(G)+A(G)$ respectively, where $A(G)$ is the adjacent matrix and $D(G)$ is the diagonal matrix of vertex degrees of $G$. It is well known that both $L(G)$ and $Q(G)$ are symmetric and positive semidefinite, then we can denote the eigenvalues of $L(G)$ and $Q(G)$, called respectively the Laplacian eigenvalues and the signless Laplacian eigenvalues of $G$, by $\mu_{1}(G)\geq\mu_{2}(G)\geq\ldots\geq\mu_{n}(G)=0$ and $q_{1}(G)\geq q_{2}(G)\geq\ldots\geq q_{n}(G)\geq 0$. Let $|U|$ be the cardinality of a finite set $U.$

Grone and Merris \cite{1994} conjectured that for a graph with $n$ vertices and degree sequence $\{d_v| v\in V(G)\},$ the following holds:
\vskip0.1cm
\hskip.5cm  $S_k(G)=\sum\limits_{i=1}^k\mu_i(G)\leq\sum\limits_{i=1}^k|\{v\in V(G)| d_v\geq i\}|,$ \quad for $k=1, 2, \ldots, n.$
\vskip0.1cm
Recently, it was proved by Bai \cite{2011a}. As a variation of the Grone-Merris conjecture, Brouwer \cite{2012c} conjectured that for a graph $G$ with $n$ vertices,
\vskip0.1cm
\hskip.5cm $S_k(G)\leq e(G)+\binom{k+1}{2},$ \quad for $k=1, 2, \ldots, n.$
\vskip0.1cm
This conjecture attracted many researchers. Haemers et al. \cite{2010b} showed that it is true for $k=2.$ Moreover, they obtained an upper bound of $S_k(T)$ for any tree $T$ with $n$ vertices and $k=1, 2, \ldots, n,$ i.e.,

\hskip.5cm $S_k(T)\leq e(T)+2k-1\leq e(T)+\binom{k+1}{2},$
\vskip0.1cm
\noindent and proved it holds for trees and threshold graphs. This was improved in \cite{2011b} to the stronger equality as follows:

\hskip.5cm $S_k(T)\leq e(T)+2k-1-\frac{2k-2}{n},$ \quad for $k=1, 2, \ldots, n.$
\vskip0.1cm
\noindent Moreover, the conjecture was proved to be true for unicyclic graphs, bicyclic graphs, regular graphs, split graphs, cographs and graphs with at most ten vertices. For more details, we refer readers to the references  \cite{2012d}, \cite{2012c},  \cite{2012b}, \cite{2010b}, \cite{2010c}, \cite{2012a}.

Motivated by the definition of $S_k(G)$ and Brouwer's conjecture, F. Ashraf et al. \cite{2013} proposed the following conjecture about $S^+_k(G),$ where $S^+_k(G)=\sum_{i=1}^kq_i(G)$ for $k=1, 2, \ldots, n.$

\begin{con}\label{con11} For any graph $G$ with $n$ vertices and any $k=1, 2, \ldots, n,$

\vskip0.2cm
\hskip0.5cm $S^+_k(G)\leq e(G)+\binom{k+1}{2}.$
\end{con}

In \cite{2013}, the authors proved that Conjecture \ref{con11} is true for $k=1,2,n-1,n$ for all graphs, for all $k$ for regular graphs and for all graphs with at most ten vertices.


In this paper, some useful notations and preliminaries are given in Section 2. By employing similar techniques to those applied in \cite{2012b} or \cite{2012a}, we give various upper bounds for $S^+_k(G)$ and show Conjecture \ref{con11} is true for connected graph with sufficiently large $k$ in Section 3, and prove Conjecture \ref{con11} to be true in Section 4 for the following cases: unicyclic graphs and bicyclic graphs for all $k,$ and tricyclic graphs when $k\neq3$.

\section{Preliminaries}
\hskip.6cm  In this section, we introduce some notations and basic properties which we need to use in the proofs of our main results.

For two graphs $G$ and $H$, the union of $G$ and $H$, denoted $G\cup H,$ is the graph whose vertex set is $V(G)\cup V(H)$ and whose edge set is $E(G)\cup E(H).$
Let $nG$ denote $n$ copies of a graph $G.$  For $H,$ a subgraph of $G,$ $G-E(H)$ is the subgraph of $G$ whose edge set is $E(G)\setminus E(H),$ while $G-H$ is the subgraph of $G$ induced by the vertex set $V(G)\setminus V(H).$

Let $\lambda_1(M)\geq \lambda_2(M)\geq \cdots\geq\lambda_n(M)$ be the eigenvalues of any matrix $M,$ and 
$\sigma(M)=\{\lambda_1(M), \lambda_2(M), \cdots, \lambda_n(M)\}$ be the spectrum of M.

\begin{lem}\label{lem24}{\rm (\cite{1949})}
Let $A$ and $B$ be two real symmetric matrices of order $n.$ Then for any $1\leq k\leq n,$ $\sum\limits_{i=1}^k\lambda_i(A+B)\leq \sum\limits_{i=1}^k\lambda_i(A)+\sum\limits_{i=1}^k\lambda_i(B).$
\end{lem}

For convenience, if $k>n,$ we denote $S^+_k(G)=S^+_n(G)$ since $S^+_k(G\cup (k-n)K_1)=S^+_n(G).$

\begin{lem}\label{lem25}
Let $G$ be a graph with $n$ vertices, $G_1, G_2, \ldots, G_t$ be the edge disjoint subgraphs of $G$ such that $E(G)=\cup_{i=1}^{t}E(G_i),$ where $t\geq 1.$ Then for any integer $k$ with $1\leq k\leq n,$
\vskip.1cm
\hskip.5cm $S^+_k(G)\leq\sum\limits_{i=1}^tS^+_k(G_i).$
\end{lem}

\begin{proof}
Let $|V(G_i)|=n_i,$ $i=1, 2, \ldots, t.$ By Lemma \ref{lem24}, we have

\hskip.5cm $S_k^+(G)=\sum\limits_{j=1}^kq_j(G)=\sum\limits_{j=1}^k\lambda_j(Q(G))$

\hskip1.75cm $=\sum\limits_{j=1}^k\lambda_j(\sum\limits_{i=1}^tQ(G_i\cup(n-n_i)K_1))$

\hskip1.75cm $\leq\sum\limits_{i=1}^t\sum\limits_{j=1}^k\lambda_j(Q(G_i\cup(n-n_i)K_1))$

\hskip1.75cm $=\sum\limits_{i=1}^tS^+_k(G_i).$
\end{proof}

\begin{cor}\label{cor26}If Conjecture \ref{con11} is false for some integer $k,$ then there exits a counterexample $G$ such that $S^+_k(H)> e(H)$ for every nonempty subgraph $H$ of $G.$
\end{cor}
\begin{proof}
Let $G$ be a counterexample of Conjecture \ref{con11} for some integer $k,$ having the minimun number of edges. If $G$ has a nonempty subgraph $H$ with $S^+_k(G)\leq e(H),$ then by Lemma \ref{lem25},
\vskip0.1cm
\hskip.5cm $e(G)+\binom{k+1}{2}<S^+_k(G)\leq S^+_k(H)+S^+_k(G-E(H))\leq e(H)+S^+_k(G-E(H)).$
\vskip0.1cm
\noindent This implies that $S^+_k(G-E(H))>e(G)-e(H)+\binom{k+1}{2}=e(G-E(H))+\binom{k+1}{2},$ which is a contradiction to the choice of $G.$
\end{proof}

\begin{rem}\label{rem27}
Corollary \ref{cor26} is a natural extension of Lemma 15 in \cite{2013}.
\end{rem}

\begin{lem}\label{lem22}{\rm (\cite{2013})}
If for some $k,$ Conjecture \ref{con11} holds for graphs $G$ and $H$, then it holds for $G\cup H.$
\end{lem}
Therefore, in order to prove Conjecture \ref{con11}, it suffices to do so for connected graphs. Thus we only need to consider the following conjecture.

\begin{con}\label{con21}
For any connected graph $G$ with $n$ vertices and any $k=1, 2, \ldots, n,$
\vskip0.2cm
\hskip.5cm $S^+_k(G)\leq e(G)+\binom{k+1}{2}.$
\end{con}

\section{Upper bounds for $S^+_k(G)$}
\hskip.6cm In this section, we give various upper bounds for $S^+_k(G)$ in terms of the clique number $\omega$ and $e(G),$ the maximum degree $\Delta$ and $e(G),$ the matching number $m$ and $e(G)$, respectively. Furthermore, we also show Conjecture \ref{con21} is true for connected graph with sufficiently large $k$, which plays an important role in the proofs of our main results.

Recall that the clique number of $G$ is the number of vertices of a maximum complete subgraph of $G.$ A matching $M$ of $G$ is a subset of $E(G)$ such that no two edges in $M$ share a common vertex. A maximum matching is a matching which covers as many vertices as possible. The matching number of $G$ is the number of edges in a maximum matching of $G.$

\begin{thm}\label{thm31}
Let $G$ be a graph with clique number $\omega.$ Then for any $k=1, 2, \ldots, \omega,$
\vskip0.2cm
\hskip.5cm $S^+_k(G)\leq2e(G)-\omega^2+(k+2)\omega-2k.$
\end{thm}
\begin{proof}
Obviously, $K_\omega$ is a subgraph of $G.$ Note that $\sigma(Q(K_\omega))=\{2\omega-2, \omega-2^{[\omega-1]}\},$ where $\lambda^{[t]}$
means that $\lambda$ is an eigenvalue with multiplicity $t$. Thus, by Lemma \ref{lem25},

\hskip.5cm $S^+_k(G)\leq S^+_k(K_\omega)+[e(G)-e(K_\omega)]S^+_k(K_2)$

\hskip1.8cm $=2\omega-2+(k-1)(\omega-2)+2[e(G)-\binom{\omega}{2}]$

\hskip1.8cm $=2e(G)-\omega^2+(k+2)\omega-2k.$
\end{proof}

\begin{thm}\label{thm32}
Let $G$ be a graph with maximum degree $\Delta.$ Then for  any $k=1, 2, \ldots, \Delta,$
\vskip0.2cm
\hskip.5cm $S^+_k(G)\leq2e(G)-\Delta+k.$
\end{thm}

\begin{proof}
Obviously, $K_{1,\Delta}$ is a subgraph of $G.$ Note that $\sigma(Q(K_{1,\Delta}))=\{\Delta+1, 1^{[\Delta-1]}, 0\}.$  Thus, by Lemma \ref{lem25},

$S^+_k(G)\leq S^+_k(K_{1,\Delta})+[e(G)-e(K_{1,\Delta})]S^+_k(K_2)=\Delta+k+2[e(G)-\Delta]=2e(G)-\Delta+k.$
\end{proof}

\begin{thm}\label{thm33}
Let $G$ be a graph with matching number $m.$ Then for any $k=1, 2, \ldots, m,$
\vskip0.2cm
\hskip.5cm $S^+_k(G)\leq2e(G)-2m+2k.$
\end{thm}

\begin{proof}
Obviously, $mK_2$ is a subgraph of $G.$ Note that $\sigma(Q(mK_2))=\{2^{[m]}, 0^{[m]}\}.$  Thus, by Lemma \ref{lem25},
$S^+_k(G)\leq S^+_k(mK_2)+[e(G)-e(mK_2)]S^+_k(K_2)=2e(G)-2m+2k.$
\end{proof}

\begin{lem}\label{lem23}{\rm (\cite{2010a})}
If $G$ is bipartite, then $Q(G)$ and $L(G)$ share the same eigenvalues.
\end{lem}

By Lemma \ref{lem23},  Conjecture \ref{con21} is true for trees follows from  Brouwer's conjecture is true for trees, that is, for any tree $T$ with $n$ vertices and any $k=1, 2, \ldots, n,$
\vskip0.2cm
\hskip.5cm $S^+_k(T)\leq e(T)+2k-1-\frac{2k-2}{n}< e(T)+\binom{k+1}{2}.$  \quad \quad \quad\quad \quad \quad\quad \quad \quad \quad \quad \quad\quad \quad \quad\quad\quad(1)

\begin{thm}\label{thm34}
Let $G$ be a connected graph with $n$ vertices. Then
\vskip0.2cm
\hskip.5cm $S^+_k(G)\leq 2e(G)+2k-n-\frac{2k-2}{n}.$
\end{thm}

\begin{proof}
Note that $G$ is connected. Thus $G$ have a spanning tree, denote by $T.$ From the Inequality (1) and Lemma \ref{lem25}, we have

\hskip.5cm $S^+_k(G)\leq S^+_k(T)+[e(G)-e(T)]S^+_k(K_2)$
\vskip0.1cm
\hskip1.8cm $\leq e(T)+2k-1-\frac{2k-2}{n}+2[e(G)-e(T)]$
\vskip0.1cm
\hskip1.8cm $=2e(G)+2k-n-\frac{2k-2}{n}.$
\end{proof}

\begin{cor}\label{cor35}
Let $G$ be a connected graph with $n$ vertices. Then for an integer $k$ with $\frac{3n-4+\sqrt{8n^2e(G)-8n^3+9n^2-8n+16}}{2n}\leq k\leq n,$ we have
\vskip0.2cm
\hskip.5cm $S^+_k(G)\leq e(G)+\binom{k+1}{2}.$
\end{cor}

\begin{proof}
Note that  $\frac{3n-4+\sqrt{8n^2e(G)-8n^3+9n^2-8n+16}}{2n}\leq k\leq n,$ which implies that $e(G)\leq n+\frac{2k-2}{n}+\frac{k^2-3k}{2}.$ Then by Theorem \ref{thm34},

\hskip.5cm $S^+_k(G)\leq 2e(G)+2k-n-\frac{2k-2}{n}$

\hskip1.75cm $\leq e(G)+n+\frac{2k-2}{n}+\frac{k^2-3k}{2}+2k-n-\frac{2k-2}{n}$

\hskip1.75cm $=e(G)+\binom{k+1}{2}.$
\end{proof}

\begin{rem}\label{rem36}
Corollary \ref{cor35} shows that for connected graphs with $n$ vertices, Conjecture \ref{con21} is true when $k$ is sufficiently large.
\end{rem}

\begin{thm}\label{thm37}
Let $n,k$ be positive integers with $1\leq k\leq n,$ and $G$ be a graph with $n$ vertices and without isolated vertices. Then
\vskip0.2cm
\hskip.5cm $S^+_k(G)\leq 2e(G)+2k-n.$
\end{thm}

\begin{proof}
If $G$ is connected, then by Theorem \ref{thm34}, $S_k^+(G)\leq 2e(G)+2k-n-\frac{2k-2}{n}< 2e(G)+2k-n.$

Now suppose that $G$ is not connected. Let $G_1, G_2,\ldots,G_t$ be all the components of $G.$ Suppose that $k_i$ of the $k$ largest signless
Laplacian eigenvalues of $\cup_{i=1}^tG_i$ come from $\sigma(Q(G_i)),$ where $i=1,2,\ldots,t$ and $\sum_{i=1}^tk_i=k.$
Without loss of generality, suppose that $k_1,k_2,\ldots,k_r>0$ and $k_{r+1}=k_{r+2}=\ldots=k_t=0$ where $1\leq r\leq t.$
Then $S^+_k(G)=\sum_{i=1}^rS^+_k(G_i).$ Let $H=\cup_{i=1}^rG_i$ and $n_i=|V(G_i)|$ for $i=1,2,\ldots,t.$ Clearly,

\hskip.5cm $S^+_k(G)=S^+_k(H)=\sum\limits_{i=1}^rS^+_{k_i}(G_i)\leq\sum\limits_{i=1}^r[2e(G_i)+2k_i-n_i]=2e(H)+2k-|V(H)|.$



\noindent Note that $e(G_i)\geq1$ for $r+1\leq i\leq t$ since $G$ has no isolated vertices. Then for  $r+1\leq i\leq t,$

\hskip.5cm $2e(G_i)-n_i\geq e(G_i)+(n_i-1)-n_i=e(G_i)-1\geq0,$

\noindent which implies that $2e(G)-n\geq 2e(H)-|V(H)|.$ Thus, $S^+_k(G)\leq 2e(G)+2k-n.$
\end{proof}

\section{Conjecture \ref{con21} for unicyclic, bicyclic and tricyclic graphs }
\hskip.6cm
In this section, we prove that Conjecture \ref{con21} is true for unicyclic and bicyclic graphs with $n$ vertices for all integer $k$, 
and  for tricyclic graphs with $n$ vertices when $k\neq 3,$ where $1\leq k\leq n.$

\begin{thm}\label{thm41}
Let $n, k$ be positive integers with  $1\leq k\leq n,$ and $G$ be an unicyclic graph with $n$ vertices. Then
\vskip0.2cm
\hskip.5cm $S^+_k(G)\leq e(G)+\binom{k+1}{2}.$
\end{thm}
\begin{proof}
The cases of $k=1$ and $k=2$ have been proved in \cite{2013}. Now we show $k\geq3$ holds.

Since $G$ is an unicyclic graph, $e(G)=n$. Then
\vskip0.1cm
\hskip.5cm $\frac{3n-4+\sqrt{8n^2e(G)-8n^3+9n^2-8n+16}}{2n}=\frac{3n-4+\sqrt{9n^2-8n+16}}{2n}< 3\leq k.$
\vskip0.1cm

 Then by Corollary \ref{cor35}, we have $S^+_k(G)\leq e(G)+\binom{k+1}{2}$ for $k=1, 2, \ldots, n.$
\end{proof}

\begin{thm}\label{thm42}
Let $n,k$ be positive integers with $1\leq k\leq n$ and $k\neq3,$ and $G$ be a bicyclic (respectively tricyclic) graph with $n$ vertices. Then
\vskip0.2cm
\hskip.5cm $S^+_k(G)\leq e(G)+\binom{k+1}{2}.$
\end{thm}
\begin{proof}
The cases of $k=1$ and $k=2$ have been proved in \cite{2013}. Now we show $k\geq4$ holds.

Since $G$ is a bicyclic (respectively tricyclic) graph, $e(G)=n+1$ (respectively $e(G)=n+2$). Then
\vskip0.1cm
\hskip0.5cm $\frac{3n-4+\sqrt{17n^2-8n+16}}{2n}<\frac{3n-4+\sqrt{25n^2-8n+16}}{2n}< 4\leq k.$
\vskip0.1cm

Then by Corollary \ref{cor35}, we have $S^+_k(G)\leq e(G)+\binom{k+1}{2}$ for  $1\leq k\leq n$ with $k\neq3.$
\end{proof}

By Theorem \ref{thm42} and the fact that Conjecture \ref{con21} is true for all graphs with at most ten vertices, to show Conjecture \ref{con21} is true for bicyclic graphs with $n$ vertices,
we only need to show that it is true for bicyclic graphs for $k=3$ and $n\geq11.$ The following notations, put forward in \cite{2012a}, and lemmas are needed in the proofs of
our main results.

Let $G_1\sim G_2$ denote the graph obtained from $G_1$ and $G_2$ by connecting a vertex of $G_1$ with a vertex of $G_2.$ Let $G_1\approx G_2$ denote the graph obtained from $G_1$ and $G_2$ by inserting two edges between $V(G_1)$ and $V(G_2).$ The following two lemmas show that if Conjecture \ref{con21} holds for $G_1$ and $G_2,$ then the Conjecture \ref{con21} is also true for $G_1\sim G_2$ and $G_1\approx G_2.$

\begin{lem}\label{lem43}
Let $G_i$  be a nonempty connected graph with $n_i$ vertices, where $i=1,2.$ If $S^+_k(G_i)\leq e(G_i)+\binom{k_i+1}{2}$ for $k_i=1,2,\ldots, n_i$ and $i=1,2,$ then for  $k=1, 2, \ldots, n_1+n_2,$
\vskip0.2cm
\hskip.5cm $S^+_k(G_1\sim G_2)\leq e(G_1\sim G_2)+\binom{k+1}{2}.$
\end{lem}

\begin{proof}
Assume that $k_i$ of the $k$ largest signless Laplacian eigenvalues of $G_1\cup G_2$ come from $\sigma(Q(G_i)),$ where $i=1,2$ and $k_1+k_2=k.$ Since $G_i$ is nonempty, $e(G_i)\geq1$ and $e(G_1\sim G_2)=e(G_1)+e(G_2)+1\geq e(G_i)+2$ for $i=1,2.$

{\bf Case 1: } $k_1k_2=0.$

Without loss of generality, suppose that $k_2=0,$ then $k_1=k$. By Lemma \ref{lem25},

\hskip.5cm $S^+_k(G_1\sim G_2)\leq S_k^+(G_1\cup G_2)+S^+_k(K_2)$
\vskip0.1cm
\hskip3cm $=S^+_k(G_1)+S^+_k(K_2)$
\vskip0.1cm
\hskip3cm $\leq (e(G_1)+\binom{k_1+1}{2})+2$
\vskip0.1cm
\hskip3cm $\leq e(G_1\sim G_2)+\binom{k+1}{2}.$

{\bf Case 2: } $k_1k_2\neq0.$

Clearly, $k_1k_2\geq1.$ Then by Lemma \ref{lem25},

\hskip.5cm $S^+_k(G_1\sim G_2)\leq S^+_k(G_1\cup G_2)+S^+_k(K_2)$

\hskip3cm $=\sum\limits_{i=1}^2S^+_{k_i}(G_i)+S^+_k(K_2)$

\hskip3cm $\leq\sum\limits_{i=1}^2[e(G_i)+\binom{k_i+1}{2}]+2$

\hskip3cm $=e(G_1\sim G_2)+\frac{k_1^2+k_2^2+k_1+k_2+2}{2}$
\vskip0.1cm
\hskip3cm $\leq e(G_1\sim G_2)+\frac{(k_1+k_2)^2+k_1+k_2}{2}$
\vskip0.1cm
\hskip3cm $= e(G_1\sim G_2)+\binom{k+1}{2}.$
\end{proof}

Let $T$ be an induced subtree of $G.$ If $G$ can be formed by connecting a vertex of $T$ with a vertex of $G-T,$ then $T$ is called a hanging tree.

\begin{cor}\label{cor44}
Let $G$ be a graph with $n(\geq4)$ vertices such that there is a hanging tree $T$ of $G$ with $|V(T)|\geq2.$ Suppose $|V(G-T)|=s,$
where $2\leq s\leq n-2.$ If $G-T$ is nonempty and $S^+_{l}(G-T)\leq e(G-T)+\binom{l+1}{2}$ for $l=1,2,\ldots,s,$ then for $k=1,2,\ldots,n,$
\vskip0.2cm
\hskip0.5cm $S^+_k(G)\leq e(G)+\binom{k+1}{2}.$
\end{cor}

\begin{proof}
It is clearly that $G=(G-T)\sim T.$ Then by Inequality (1) and Lemma \ref{lem43}, the result holds.
\end{proof}

\begin{lem}\label{lem44}
Let $G_i$  be a graph with $n_i$ vertices and $e(G_i)\geq 2,$ where $i=1,2$. If $S^+_k(G_i)\leq e(G_i)+\binom{k_i+1}{2}$ for $k_i=1,2,\ldots, n_i$
and $i=1,2,$ then for $k=1,2,\ldots,n_1+n_2,$
\vskip.2cm
\hskip.5cm $S^+_k(G_1\approx G_2)\leq e(G_1\approx G_2)+\binom{k+1}{2}.$
\end{lem}

\begin{proof}
The cases of $k=1$ is trivial, and the case of $k=2$ have been confirmed in \cite{2013}. We assume that $k\geq3$ in the following, and $k_i$ of the $k$ largest signless Laplacian eigenvalues of $G_1\cup G_2$ come from $\sigma(Q(G_i)),$ where $i=1,2$ and $k_1+k_2=k.$ Since $e(G_i)\geq2$, we have $e(G_1\approx G_2)=e(G_1)+e(G_2)+2\geq e(G_i)+4$ for $i=1,2.$

{\bf Case 1: } $k_1k_2=0.$

Without loss of generality, suppose $k_2=0,$ then $k_1=k.$ Hence, by Lemma \ref{lem25},

\hskip.5cm $S^+_k(G_1\approx G_2)\leq S_k^+(G_1\cup G_2)+2S^+_k(K_2)$
\vskip0.1cm
\hskip3cm $=S^+_k(G_1)+2S^+_k(K_2)$
\vskip0.1cm
\hskip3cm $\leq e(G_1)+\binom{k_1+1}{2}+4$
\vskip0.1cm
\hskip3cm $\leq e(G_1\approx G_2)+\binom{k+1}{2}.$

{\bf Case 2: } $k_1k_2\neq0.$

Since $k_1k_2\neq0$ and $k\geq3,$ $k_1k_2\geq2.$ Then by Lemma \ref{lem25},

\hskip.5cm $S^+_k(G_1\approx G_2)\leq S^+_k(G_1\cup G_2)+2S^+_k(K_2)$

\hskip3cm $=\sum\limits_{i=1}^2S^+_{k_i}(G_i)+2S^+_k(K_2)$

\hskip3cm $\leq\sum\limits_{i=1}^2[e(G_i)+\binom{k_i+1}{2}]+4$

\hskip3cm $=e(G_1\approx G_2)+\frac{k_1^2+k_2^2+k_1+k_2+4}{2}$
\vskip0.1cm
\hskip3cm $\leq e(G_1\approx G_2)+\frac{(k_1+k_2)^2+k_1+k_2}{2}$
\vskip0.1cm
\hskip3cm $= e(G_1\approx G_2)+\binom{k+1}{2}.$
\end{proof}

Let $\psi(G,x)$ be the signless Laplacian characteristic polynomial of $G,$ equal to $det(xI-Q(G)).$ Let $Q_v(G)$ be the principal submatrix of $Q(G)$ obtained by deleting the row and column corresponding to the vertex $v.$

\begin{lem}\label{lem26}{\rm (\cite{2012e}, \cite{2012f})}
Let $G$ be a connected graph with $n$ vertices which consists of a subgraph $H$(with at least two vertices) and $n-|H|$ distinct pendant vertices (not in $H$) attaching to a vertex $v$ in $H.$ Then
\vskip0.2cm
\hskip.5cm $\psi(G,x)=(x-1)^{n-|H|}\psi(H,x)-(n-|H|)x(x-1)^{n-|H|-1}\psi(Q_v(H)).$
\end{lem}

Let $\infty(p,q,1)$ denote the bicyclic graph obtained from two cycles $C_{p},$ $C_{q}$ by identifying a vertex of $C_{p}$ with a vertex of $C_{q},$ where $q\geq p\geq 3.$ Let $\infty(p,q,t)$ denote the bicyclic graph obtained from a path $P_t$ and two cycles $C_{p},$ $C_{q}$ by identifying a vertex of $C_{p}$ with one end of $P_t$ and a vertex of $C_{q}$ with the other end vertex of $P_t,$ where $q\geq p\geq 3$ and $t\geq2.$ (see Fig.1)

\vspace{-0.25cm}
\begin{picture}(500,100)
\put(50,50){\circle{50}}
\put(90,50){\circle{50}}
\put(70,50){\circle*{5}}

\put(45,45){$C_p$}
\put(85,45){$C_q$}
\put(42,12){$\infty(p,q,1)$}

\put(240,50){\circle{50}}
\put(377,50){\circle{50}}

\put(260,50){\circle*{5}}
\put(260,50){\line(1,0){30}}
\put(290,50){\circle*{5}}
\put(298,47){$\cdots$}
\put(326,50){\circle*{5}}
\put(326,50){\line(1,0){30}}
\put(356,50){\circle*{5}}

\put(261,55){$\overbrace{\quad\quad\quad\quad\quad\quad\quad\quad}$}
\put(301,70){\small $P_t$}
\put(235,45){\small $C_p$}
\put(370,45){\small $C_q$}
\put(270,12){\small $\infty(p,q,t)(t\geq2)$}
\put(80,-13){\small Fig.1\quad The graphs $\infty(p,q,1)$ and $\infty(p,q,t)(t\geq2)$}
\end{picture}

\vskip0.8cm
A bicyclic graph is called $\infty$-type if it is $\infty(p,q,t)$ or it can be obtained by attaching some hanging trees to $\infty(p,q,t),$ where $q\geq p\geq 3$ and $t\geq1.$ We will show that Conjecture \ref{con21} is true for $\infty$-type bicyclic graphs when $k=3$ and $n\geq  11.$

\begin{thm}\label{thm46}
Let $n,k$ be integers with $n\geq11$ and $1\leq k\leq n,$ and $G$ be an $\infty$-type bicyclic graph with $n$ vertices. Then
\vskip.2cm
\hskip.5cm $S^+_k(G)\leq e(G)+\binom{k+1}{2}.$
\end{thm}
\begin{proof}
From Theorem \ref{thm42}, we only need to prove the case of $k=3.$

Let $G$ be an $\infty$-type bicyclic graph with $n$ vertices, then $G$ is obtianed by attaching some hanging trees to $\infty(p,q,t,)$ where $q\geq p\geq3$ and $t\geq1.$ By Corollary \ref{cor44}, it will suffice to consider the $\infty$-type bicyclic graph $G$ which is obtained by attaching some pendent vertices to $\infty(p,q,t)$ or $G\cong\infty(p,q,t).$

{\bf Case 1:} $t\geq2.$

Let $e$ be an edge of $P_t$ in  $\infty(p,q,t),$ then $G-e$ is the union of two unicyclic graphs. By Theorem \ref{thm41}
and Lemma \ref{lem43}, we get the desired result.

{\bf Case 2:} $t=1.$

Let $\infty'(3,3,1)$ be the graph obtained from $\infty(3,3,1)$ by attaching $n-5$ pendent vertices to the common vertex of two cycles.
If $G\cong \infty'(3,3,1),$ then by Lemma \ref{lem26} and direct calculation, $\psi(\infty'(3,3,1),x)=(x-1)^{n-4}(x-3)[x^3-(n+3)x^2+3nx-8].$
\noindent Thus $S_3^+(\infty'(3,3,1))<(n+3)+3+1=(n+1)+\binom{3+1}{2}=e(\infty'(3,3,1))+\binom{3+1}{2}.$

Otherwise, there exits two edges $e_1$ and $e_2$ of a cycle such that $G-\{e_1,e_2\}$ is the union of a unicyclic graph and a tree with at least two edges. Combining the Inequality (1), Theorem \ref{thm41} and Lemma \ref{lem44}, the result is obtained.
\end{proof}

Let $p,q,t$ be integers with $p, q\geq3,$ and $2\leq t\leq\min\{\frac{p+2}{2},\frac{q+2}{2}\},$ $\theta(p,q,t)$ denote the bicyclic graph in which two cycles $C_{p}$ and $C_{q}$ share path $P_t.$ Clearly, the cycle $C_{p+q-2t+2}$ has the maximum length among three cycles $C_p,$ $C_q,$ $C_{p+q-2t+2}.$ A bicyclic graph is called $\theta$-type if it is $\theta(p,q,t)$ or it can be obtained by attaching some hanging trees to $\theta(p,q,t).$ We will show that Conjecture \ref{con21} is true for $\theta$-type bicyclic graphs when $k=3$ and $n\geq11.$  The following lemmas are essential in our proof.

Let $U_n^1(a,b)$ be the graph obtained by attaching $a$ and $b$ pendent vertices to two vertices of a triangle, respectively, where $a+b=n-3$, $n\geq4$ and $a\geq b\geq0.$ Let $U_n^2(a,b)$ be the graph obtained by attaching $a$ and $b$ pendent vertices to two nonadjacent vertices of a quadrangle, respectively, where $a+b=n-4$, $n\geq5$ and $a\geq b\geq0.$ (see Fig.2)

\vspace{0.25cm}
\begin{picture}(500,100)
\put(155,30){\circle*{5}}
\put(125,45){\circle*{5}}
\put(125,15){\circle*{5}}

\put(125,45){\line(2,-1){30}}
\put(125,15){\line(2,1){30}}
\put(125,45){\line(0,-1){30}}

\put(125,45){\line(-1,2){15}}
\put(125,45){\line(0,1){30}}
\put(125,45){\line(1,1){30}}

\put(110,75){\circle*{5}}
\put(125,75){\circle*{5}}
\put(155,75){\circle*{5}}
\put(131,72){$\cdots$}

\put(125,15){\line(-1,-2){15}}
\put(125,15){\line(0,-1){30}}
\put(125,15){\line(1,-1){30}}

\put(110,-15){\circle*{5}}
\put(125,-15){\circle*{5}}
\put(155,-15){\circle*{5}}
\put(131,-18){$\cdots$}

\put(128,86){$a$}
\put(128,-34){$b$}
\put(105,77){ $\overbrace{\quad \quad \quad \quad}$}
\put(108,-17){$\underbrace{\quad \quad \quad \quad}$}
\put(113,-48){\small $U^1_n(a,b)$}

\put(260,30){\circle*{5}}
\put(275,45){\circle*{5}}
\put(275,15){\circle*{5}}
\put(305,30){\circle*{5}}

\put(260,30){\line(1,1){15}}
\put(260,30){\line(1,-1){15}}

\put(275,45){\line(2,-1){30}}
\put(275,15){\line(2,1){30}}

\put(275,45){\line(-1,2){15}}
\put(275,45){\line(0,1){30}}
\put(275,45){\line(1,1){30}}

\put(260,75){\circle*{5}}
\put(275,75){\circle*{5}}
\put(305,75){\circle*{5}}
\put(281,72){$\cdots$}

\put(275,15){\line(-1,-2){15}}
\put(275,15){\line(0,-1){30}}
\put(275,15){\line(1,-1){30}}

\put(260,-15){\circle*{5}}
\put(275,-15){\circle*{5}}
\put(305,-15){\circle*{5}}
\put(281,-18){$\cdots$}

\put(278,86){$a$}
\put(278,-34){$b$}
\put(255,77){ $\overbrace{\quad \quad \quad \quad}$}
\put(258,-17){$\underbrace{\quad \quad \quad \quad}$}
\put(263,-48){\small $U^2_n(a,b)$}
\put(105,-68){\small Fig.2\quad The graphs $U^1_n(a,b)$ and $U^2_n(a,b).$}
\end{picture}

\vspace{2.5cm}
\begin{lem}\label{lem47}
For $n\geq5,$ $a\geq b\geq0$ and $a\geq2,$ $q_3(U^1_n(a,b))<2.$
\end{lem}

\begin{proof}
By Lemma \ref{lem26} and direct calculation, we have

\hskip0.5cm $\psi(U^1_n(a,b),x)=(x-1)^{a+b+2}f(x),$

\noindent where $f(x)=x^5-(n+5)x^4+[(a+5)n-a^2-3a+7]x^3-[(2a+7)n-2a^2-6a+7]x^2+(3n+8)x-4.$

 Let $x_1\geq x_2\geq x_3\geq x_4\geq x_5$ be the roots of $f(x)=0.$ Note that

\hskip.5cm $f(0)=-4<0,$ $f(1)=-ab\leq0,$ $f(2)=2a+2b-2>0,$

\hskip.5cm $f(a+2)=-a^4+a^3b+3a^2b-3a^3-5a^2+4ab-5a+3b-2$

\hskip2.1cm $\leq -a^4+a^4+3a^3-3a^3-5a^2+4a^2-5a+3a-2$

\hskip2.1cm $=-a^2-2a-2<0,$

and

\hskip.5cm$f(\frac{2}{n})=\frac{2}{n^5}[n^5-(4a+6)n^4+(4a^2+16a+6)n^3-(4a^2+12a-20)n^2-40n+16]$
\vskip0.1cm
\hskip1.45cm $=\frac{2}{n^5}\{n^3[n-(2a+3)]^2+(4a-3)n^3-(4a^2+12a-20)n^2-40n+16\}$
\vskip0.1cm
\hskip1.45cm $=\frac{2}{n^5}\{n^2[(n-2a-3)^2n+(4a-3)n-4a^2-12a+9]+11n^2-40n+16\}$
\vskip0.1cm
\hskip1.45cm $=\frac{2}{n^5}\{n^2[(b-a)^2(a+b+3)+4ab-3a-3b]+11n^2-40n+16\}$
\vskip0.1cm
\hskip1.45cm $=\frac{2}{n^5}\{n^2[(a-b)^2(a+b)+3a^2+3b^2-2ab-3a-3b]+11n^2-40n+16\}.$
\vskip0.1cm
{\bf Case 1:} $a=2.$
\vskip0.1cm
\hskip.5cm $f(\frac{2}{n})=\frac{2}{n^5}\{n^2[(2-b)^2(2+b)+3b^2-7b+6]+11n^2-40n+16\}$
\vskip0.1cm
\hskip1.45cm $=\frac{2}{n^5}\{n^2[(2-b)^2(2+b)+3(b-\frac{7}{6})^2+\frac{23}{12}]+11(n-\frac{20}{11})^2-\frac{224}{11}\}$
\vskip0.1cm
\hskip1.45cm $\geq\frac{2}{n^5}\{n^2[(2-b)^2(2+b)+3(b-\frac{7}{6})^2+\frac{23}{12}]+11(5-\frac{20}{11})^2-\frac{224}{11}\}$
\vskip0.1cm
\hskip1.45cm $=\frac{2}{n^5}\{n^2[(2-b)^2(2+b)+3(b-\frac{7}{6})^2+\frac{23}{12}]+91\}$
\vskip0.1cm
\hskip1.45cm $>0.$
\vskip0.1cm
{\bf Case 2:} $a\geq3.$
\vskip0.1cm
\hskip.5cm $f(\frac{2}{n})=\frac{2}{n^5}\{n^2[(a-b)^2(a+b)+2a(a-b)+a(a-3)+3b(b-1)]+11(n-\frac{20}{11})^2-\frac{224}{11}\}$
\vskip0.1cm
\hskip1.45cm $\geq\frac{2}{n^5}\{n^2[(a-b)^2(a+b)+2a(a-b)+a(a-3)+3b(b-1)]+11(5-\frac{20}{11})^2-\frac{224}{11}\}$
\vskip0.1cm
\hskip1.45cm $=\frac{2}{n^5}\{n^2[(a-b)^2(a+b)+2a(a-b)+a(a-3)+3b(b-1)]+91\}$
\vskip0.1cm
\hskip1.45cm $>0.$

Combining the above arguments, we have $0<x_5<\frac{2}{n}< x_4\leq 1\leq x_3<2<x_2<a+2<x_1.$ Then $q_3(U_n^1(a,b))=x_3<2.$
\end{proof}

\begin{lem}\label{lem48}{\rm (\cite{2012b})}
For $n\geq9,$ $a\geq b\geq0,$ $\mu_3(U_n^2(a,b))=2.$
\end{lem}

By Lemma \ref{lem23} and $U^2_n(a,b)$ is a bipartite graph, we have

\begin{lem}\label{lem49}
For $n\geq9, a\geq b\geq0,$ $q_3(U_n^2(a,b))=2.$
\end{lem}

\begin{lem}\label{lem410}
Let $G$ be a graph with $n$ vertices. If $G-\{e_1,e_2\}=H\cup K_2,$ where $e_1,e_2\in E(G),$ and $H=U^i_n(a,b)$ for some integers $a, b$ with $a+b=n-2-i$, $a\geq b\geq0,$ and $i=1,2,$ then $S_3^+(G)\leq e(G)+6.$
\end{lem}

\begin{proof}
By Lemma \ref{lem47} and Lemma \ref{lem49}, the first three largest signless Laplacian eigenvalues of $H\cup K_2$ are $q_1(H), q_2(H), 2,$ which implies $S^+_3(H\cup K_2)=S^+_2(H)+2.$ Hence, by Lemma \ref{lem25} and Theorem \ref{thm41},

$S^+_3(G)\leq S^+_3(H\cup K_2)+2S^+_3(K_2)=S^+_2(H)+6\leq e(H)+\binom{2+1}{2}+6=e(G)+6.$
\end{proof}

For $n\geq11,$ define sixteen classes of bicyclic graphs with $n$ vertices, denoted by $\mathbb{U}_n^1,$ $\mathbb{U}_n^2,$ $\ldots,$
$\mathbb{U}_n^{16},$ for which the structures of graphs in them are given in Fig.3. For $n\geq11,$ we also define three bicyclic graphs with $n$ vertices,
denoted by $U_n^{17},$ $U_n^{18},$ $U_n^{19},$ see also Fig.3. \\
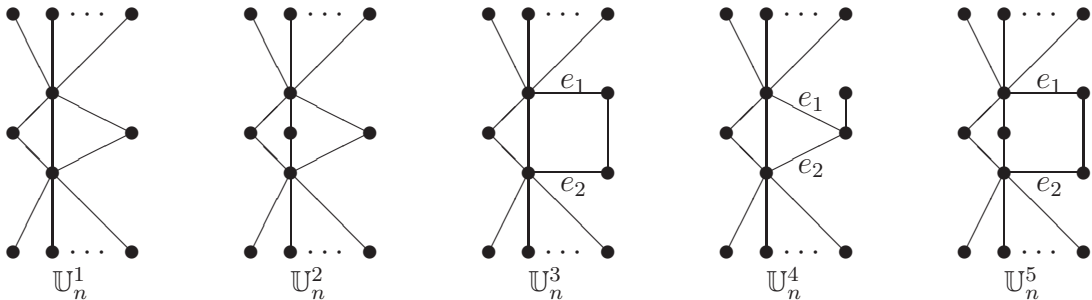
\begin{picture}(500,100)
\put(20,30){\circle*{5}}
\put(35,45){\circle*{5}}
\put(35,15){\circle*{5}}
\put(65,30){\circle*{5}}

\put(20,30){\line(1,1){15}}
\put(20,30){\line(1,-1){15}}
\put(35,45){\line(0,-1){30}}


\put(35,45){\line(2,-1){30}}
\put(35,15){\line(2,1){30}}

\put(35,45){\line(-1,2){15}}
\put(35,45){\line(0,1){30}}
\put(35,45){\line(1,1){30}}

\put(20,75){\circle*{5}}
\put(35,75){\circle*{5}}
\put(65,75){\circle*{5}}
\put(41,72){$\cdots$}

\put(35,15){\line(-1,-2){15}}
\put(35,15){\line(0,-1){30}}
\put(35,15){\line(1,-1){30}}

\put(20,-15){\circle*{5}}
\put(35,-15){\circle*{5}}
\put(65,-15){\circle*{5}}
\put(41,-18){$\cdots$}
\put(35,-30){\small $\mathbb{U}^1_n$}

\put(110,30){\circle*{5}}
\put(125,45){\circle*{5}}
\put(125,15){\circle*{5}}
\put(155,30){\circle*{5}}
\put(125,30){\circle*{5}}

\put(110,30){\line(1,1){15}}
\put(110,30){\line(1,-1){15}}
\put(125,45){\line(0,-1){30}}


\put(125,45){\line(2,-1){30}}
\put(125,15){\line(2,1){30}}

\put(125,45){\line(-1,2){15}}
\put(125,45){\line(0,1){30}}
\put(125,45){\line(1,1){30}}

\put(110,75){\circle*{5}}
\put(125,75){\circle*{5}}
\put(155,75){\circle*{5}}
\put(131,72){$\cdots$}

\put(125,15){\line(-1,-2){15}}
\put(125,15){\line(0,-1){30}}
\put(125,15){\line(1,-1){30}}

\put(110,-15){\circle*{5}}
\put(125,-15){\circle*{5}}
\put(155,-15){\circle*{5}}
\put(131,-18){$\cdots$}
\put(125,-30){\small $\mathbb{U}^2_n$}

\put(200,30){\circle*{5}}
\put(215,45){\circle*{5}}
\put(215,15){\circle*{5}}
\put(245,45){\circle*{5}}
\put(245,15){\circle*{5}}

\put(200,30){\line(1,1){15}}
\put(200,30){\line(1,-1){15}}
\put(215,45){\line(0,-1){30}}
\put(245,45){\line(0,-1){30}}

\put(227,47){$e_1$}
\put(227,8){$e_2$}
\put(215,45){\line(1,0){30}}
\put(215,15){\line(1,0){30}}

\put(215,45){\line(-1,2){15}}
\put(215,45){\line(0,1){30}}
\put(215,45){\line(1,1){30}}

\put(200,75){\circle*{5}}
\put(215,75){\circle*{5}}
\put(245,75){\circle*{5}}
\put(221,72){$\cdots$}

\put(215,15){\line(-1,-2){15}}
\put(215,15){\line(0,-1){30}}
\put(215,15){\line(1,-1){30}}

\put(200,-15){\circle*{5}}
\put(215,-15){\circle*{5}}
\put(245,-15){\circle*{5}}
\put(221,-18){$\cdots$}
\put(215,-30){\small $\mathbb{U}^3_n$}

\put(290,30){\circle*{5}}
\put(305,45){\circle*{5}}
\put(305,15){\circle*{5}}
\put(335,30){\circle*{5}}
\put(335,45){\circle*{5}}

\put(290,30){\line(1,1){15}}
\put(290,30){\line(1,-1){15}}
\put(305,45){\line(0,-1){30}}

\put(317,40){$e_1$}
\put(317,15){$e_2$}

\put(305,45){\line(2,-1){30}}
\put(305,15){\line(2,1){30}}
\put(335,30){\line(0,1){15}}

\put(305,45){\line(-1,2){15}}
\put(305,45){\line(0,1){30}}
\put(305,45){\line(1,1){30}}

\put(290,75){\circle*{5}}
\put(305,75){\circle*{5}}
\put(335,75){\circle*{5}}
\put(311,72){$\cdots$}

\put(305,15){\line(-1,-2){15}}
\put(305,15){\line(0,-1){30}}
\put(305,15){\line(1,-1){30}}

\put(290,-15){\circle*{5}}
\put(305,-15){\circle*{5}}
\put(335,-15){\circle*{5}}
\put(311,-18){$\cdots$}
\put(305,-30){\small $\mathbb{U}^4_n$}

\put(380,30){\circle*{5}}
\put(395,45){\circle*{5}}
\put(395,15){\circle*{5}}
\put(425,45){\circle*{5}}
\put(425,15){\circle*{5}}
\put(395,30){\circle*{5}}

\put(380,30){\line(1,1){15}}
\put(380,30){\line(1,-1){15}}
\put(395,45){\line(0,-1){30}}
\put(425,45){\line(0,-1){30}}

\put(407,47){$e_1$}
\put(407,8){$e_2$}
\put(395,45){\line(1,0){30}}
\put(395,15){\line(1,0){30}}

\put(395,45){\line(-1,2){15}}
\put(395,45){\line(0,1){30}}
\put(395,45){\line(1,1){30}}

\put(380,75){\circle*{5}}
\put(395,75){\circle*{5}}
\put(425,75){\circle*{5}}
\put(401,72){$\cdots$}

\put(395,15){\line(-1,-2){15}}
\put(395,15){\line(0,-1){30}}
\put(395,15){\line(1,-1){30}}

\put(380,-15){\circle*{5}}
\put(395,-15){\circle*{5}}
\put(425,-15){\circle*{5}}
\put(401,-18){$\cdots$}
\put(395,-30){\small $\mathbb{U}^5_n$}
\end{picture}

\vspace{1cm}
\begin{picture}(500,100)
\put(0,30){\circle*{5}}
\put(15,45){\circle*{5}}
\put(15,15){\circle*{5}}
\put(45,30){\circle*{5}}
\put(45,45){\circle*{5}}
\put(15,30){\circle*{5}}

\put(0,30){\line(1,1){15}}
\put(0,30){\line(1,-1){15}}
\put(15,45){\line(0,-1){30}}

\put(27,40){$e_1$}
\put(27,15){$e_2$}
\put(15,45){\line(2,-1){30}}
\put(15,15){\line(2,1){30}}
\put(45,30){\line(0,1){15}}

\put(15,45){\line(-1,2){15}}
\put(15,45){\line(0,1){30}}
\put(15,45){\line(1,1){30}}

\put(0,75){\circle*{5}}
\put(15,75){\circle*{5}}
\put(45,75){\circle*{5}}
\put(21,72){$\cdots$}

\put(15,15){\line(-1,-2){15}}
\put(15,15){\line(0,-1){30}}
\put(15,15){\line(1,-1){30}}

\put(0,-15){\circle*{5}}
\put(15,-15){\circle*{5}}
\put(45,-15){\circle*{5}}
\put(45,-15){\circle*{5}}

\put(21,-18){$\cdots$}
\put(15,-30){\small $\mathbb{U}^6_n$}

\put(90,15){\circle*{5}}
\put(90,45){\circle*{5}}
\put(120,45){\circle*{5}}
\put(120,15){\circle*{5}}
\put(150,45){\circle*{5}}
\put(150,15){\circle*{5}}

\put(90,45){\line(0,-1){30}}
\put(90,15){\line(1,0){30}}
\put(90,45){\line(1,0){30}}
\put(120,45){\line(0,-1){30}}
\put(150,45){\line(0,-1){30}}

\put(120,45){\line(1,0){30}}
\put(120,15){\line(1,0){30}}

\put(120,45){\line(-1,2){15}}
\put(120,45){\line(0,1){30}}
\put(120,45){\line(1,1){30}}

\put(105,75){\circle*{5}}
\put(120,75){\circle*{5}}
\put(150,75){\circle*{5}}
\put(126,72){$\cdots$}

\put(120,15){\line(-1,-2){15}}
\put(120,15){\line(0,-1){30}}
\put(120,15){\line(1,-1){30}}

\put(105,-15){\circle*{5}}
\put(120,-15){\circle*{5}}
\put(150,-15){\circle*{5}}
\put(126,-17){$\cdots$}
\put(110,-30){\small $\mathbb{U}^7_n$}

\put(195,30){\circle*{5}}
\put(210,45){\circle*{5}}
\put(210,15){\circle*{5}}
\put(240,45){\circle*{5}}
\put(240,15){\circle*{5}}

\put(195,30){\line(0,1){15}}
\put(195,45){\circle*{5}}

\put(195,30){\line(1,1){15}}
\put(195,30){\line(1,-1){15}}
\put(210,45){\line(0,-1){30}}
\put(240,45){\line(0,-1){30}}

\put(210,45){\line(1,0){30}}
\put(210,15){\line(1,0){30}}

\put(210,45){\line(-1,2){15}}
\put(210,45){\line(0,1){30}}
\put(210,45){\line(1,1){30}}

\put(195,75){\circle*{5}}
\put(210,75){\circle*{5}}
\put(240,75){\circle*{5}}
\put(216,72){$\cdots$}

\put(210,15){\line(-1,-2){15}}
\put(210,15){\line(0,-1){30}}
\put(210,15){\line(1,-1){30}}

\put(195,-15){\circle*{5}}
\put(210,-15){\circle*{5}}
\put(240,-15){\circle*{5}}
\put(216,-17){$\cdots$}
\put(210,-30){\small $\mathbb{U}^8_n$}

\put(285,30){\circle*{5}}
\put(300,45){\circle*{5}}
\put(300,15){\circle*{5}}
\put(330,30){\circle*{5}}
\put(330,45){\circle*{5}}
\put(285,45){\circle*{5}}

\put(285,30){\line(0,1){15}}
\put(285,30){\line(1,1){15}}
\put(285,30){\line(1,-1){15}}
\put(300,45){\line(0,-1){30}}


\put(300,45){\line(2,-1){30}}
\put(300,15){\line(2,1){30}}
\put(330,30){\line(0,1){15}}

\put(300,45){\line(-1,2){15}}
\put(300,45){\line(0,1){30}}
\put(300,45){\line(1,1){30}}

\put(285,75){\circle*{5}}
\put(300,75){\circle*{5}}
\put(330,75){\circle*{5}}
\put(306,72){$\cdots$}

\put(300,15){\line(-1,-2){15}}
\put(300,15){\line(0,-1){30}}
\put(300,15){\line(1,-1){30}}

\put(285,-15){\circle*{5}}
\put(300,-15){\circle*{5}}
\put(330,-15){\circle*{5}}
\put(306,-17){$\cdots$}
\put(300,-30){\small $\mathbb{U}^9_n$}

\put(375,15){\circle*{5}}
\put(375,45){\circle*{5}}
\put(405,45){\circle*{5}}
\put(405,15){\circle*{5}}
\put(435,45){\circle*{5}}
\put(435,15){\circle*{5}}
\put(405,30){\circle*{5}}

\put(375,45){\line(0,-1){30}}
\put(375,15){\line(1,0){30}}
\put(375,45){\line(1,0){30}}
\put(405,45){\line(0,-1){30}}
\put(435,45){\line(0,-1){30}}

\put(405,45){\line(1,0){30}}
\put(405,15){\line(1,0){30}}

\put(405,45){\line(-1,2){15}}
\put(405,45){\line(0,1){30}}
\put(405,45){\line(1,1){30}}

\put(390,75){\circle*{5}}
\put(405,75){\circle*{5}}
\put(435,75){\circle*{5}}
\put(411,72){$\cdots$}

\put(405,15){\line(-1,-2){15}}
\put(405,15){\line(0,-1){30}}
\put(405,15){\line(1,-1){30}}

\put(390,-15){\circle*{5}}
\put(405,-15){\circle*{5}}
\put(435,-15){\circle*{5}}
\put(411,-17){$\cdots$}
\put(395,-30){\small $\mathbb{U}^{10}_n$}
\end{picture}

\vspace{1cm}
\begin{picture}(500,100)
\put(0,30){\circle*{5}}
\put(15,45){\circle*{5}}
\put(15,15){\circle*{5}}
\put(45,45){\circle*{5}}
\put(45,15){\circle*{5}}
\put(15,30){\circle*{5}}

\put(0,30){\line(0,1){15}}
\put(0,45){\circle*{5}}

\put(0,30){\line(1,1){15}}
\put(0,30){\line(1,-1){15}}
\put(15,45){\line(0,-1){30}}
\put(45,45){\line(0,-1){30}}

\put(15,45){\line(1,0){30}}
\put(15,15){\line(1,0){30}}

\put(15,45){\line(-1,2){15}}
\put(15,45){\line(0,1){30}}
\put(15,45){\line(1,1){30}}

\put(0,75){\circle*{5}}
\put(15,75){\circle*{5}}
\put(45,75){\circle*{5}}
\put(21,72){$\cdots$}

\put(15,15){\line(-1,-2){15}}
\put(15,15){\line(0,-1){30}}
\put(15,15){\line(1,-1){30}}

\put(0,-15){\circle*{5}}
\put(15,-15){\circle*{5}}
\put(45,-15){\circle*{5}}
\put(21,-17){$\cdots$}
\put(15,-30){\small $\mathbb{U}^{11}_n$}

\put(90,30){\circle*{5}}
\put(105,45){\circle*{5}}
\put(105,15){\circle*{5}}
\put(135,30){\circle*{5}}
\put(135,45){\circle*{5}}
\put(90,45){\circle*{5}}
\put(105,30){\circle*{5}}

\put(90,30){\line(0,1){15}}
\put(90,30){\line(1,1){15}}
\put(90,30){\line(1,-1){15}}
\put(105,45){\line(0,-1){30}}


\put(105,45){\line(2,-1){30}}
\put(105,15){\line(2,1){30}}
\put(135,30){\line(0,1){15}}

\put(105,45){\line(-1,2){15}}
\put(105,45){\line(0,1){30}}
\put(105,45){\line(1,1){30}}

\put(90,75){\circle*{5}}
\put(105,75){\circle*{5}}
\put(135,75){\circle*{5}}
\put(111,72){$\cdots$}

\put(105,15){\line(-1,-2){15}}
\put(105,15){\line(0,-1){30}}
\put(105,15){\line(1,-1){30}}

\put(90,-15){\circle*{5}}
\put(105,-15){\circle*{5}}
\put(135,-15){\circle*{5}}
\put(111,-17){$\cdots$}
\put(105,-30){\small $\mathbb{U}^{12}_n$}

\put(180,15){\circle*{5}}
\put(180,45){\circle*{5}}
\put(210,45){\circle*{5}}
\put(210,15){\circle*{5}}
\put(240,45){\circle*{5}}
\put(240,15){\circle*{5}}
\put(210,25){\circle*{5}}
\put(210,35){\circle*{5}}

\put(180,45){\line(0,-1){30}}
\put(180,15){\line(1,0){30}}
\put(180,45){\line(1,0){30}}
\put(210,45){\line(0,-1){30}}
\put(240,45){\line(0,-1){30}}

\put(210,45){\line(1,0){30}}
\put(210,15){\line(1,0){30}}

\put(210,45){\line(-1,2){15}}
\put(210,45){\line(0,1){30}}
\put(210,45){\line(1,1){30}}

\put(195,75){\circle*{5}}
\put(210,75){\circle*{5}}
\put(240,75){\circle*{5}}
\put(216,72){$\cdots$}

\put(210,15){\line(-1,-2){15}}
\put(210,15){\line(0,-1){30}}
\put(210,15){\line(1,-1){30}}

\put(195,-15){\circle*{5}}
\put(210,-15){\circle*{5}}
\put(240,-15){\circle*{5}}
\put(216,-17){$\cdots$}
\put(200,-30){\small $\mathbb{U}^{13}_n$}

\put(285,15){\circle*{5}}
\put(285,45){\circle*{5}}
\put(315,45){\circle*{5}}
\put(315,15){\circle*{5}}
\put(345,45){\circle*{5}}
\put(345,15){\circle*{5}}
\put(315,35){\circle*{5}}

\put(315,35){\line(1,-1){10}}
\put(325,25){\circle*{5}}

\put(285,45){\line(0,-1){30}}
\put(285,15){\line(1,0){30}}
\put(285,45){\line(1,0){30}}
\put(315,45){\line(0,-1){30}}
\put(345,45){\line(0,-1){30}}

\put(315,45){\line(1,0){30}}
\put(315,15){\line(1,0){30}}

\put(315,45){\line(-1,2){15}}
\put(315,45){\line(0,1){30}}
\put(315,45){\line(1,1){30}}

\put(300,75){\circle*{5}}
\put(315,75){\circle*{5}}
\put(345,75){\circle*{5}}
\put(321,72){$\cdots$}

\put(315,15){\line(-1,-2){15}}
\put(315,15){\line(0,-1){30}}
\put(315,15){\line(1,-1){30}}

\put(300,-15){\circle*{5}}
\put(315,-15){\circle*{5}}
\put(345,-15){\circle*{5}}
\put(321,-17){$\cdots$}
\put(305,-30){\small $\mathbb{U}^{14}_n$}

\put(385,30){\line(0,1){15}}
\put(385,45){\circle*{5}}

\put(385,30){\circle*{5}}
\put(400,45){\circle*{5}}
\put(400,15){\circle*{5}}
\put(430,45){\circle*{5}}
\put(430,15){\circle*{5}}
\put(400,35){\circle*{5}}

\put(400,35){\line(1,-1){10}}
\put(410,25){\circle*{5}}

\put(385,30){\line(1,1){15}}
\put(385,30){\line(1,-1){15}}
\put(400,45){\line(0,-1){30}}
\put(430,45){\line(0,-1){30}}

\put(400,45){\line(1,0){30}}
\put(400,15){\line(1,0){30}}

\put(400,45){\line(-1,2){15}}
\put(400,45){\line(0,1){30}}
\put(400,45){\line(1,1){30}}

\put(385,75){\circle*{5}}
\put(400,75){\circle*{5}}
\put(430,75){\circle*{5}}
\put(406,72){$\cdots$}

\put(400,15){\line(-1,-2){15}}
\put(400,15){\line(0,-1){30}}
\put(400,15){\line(1,-1){30}}

\put(385,-15){\circle*{5}}
\put(400,-15){\circle*{5}}
\put(430,-15){\circle*{5}}
\put(406,-17){$\cdots$}
\put(400,-30){\small $\mathbb{U}^{15}_n$}
\end{picture}

\vspace{0.9cm}
\begin{picture}(500,100)
\put(0,30){\circle*{5}}
\put(15,45){\circle*{5}}
\put(15,15){\circle*{5}}
\put(45,30){\circle*{5}}
\put(45,45){\circle*{5}}
\put(0,45){\circle*{5}}
\put(15,30){\circle*{5}}

\put(15,30){\line(2,-1){10}}
\put(25,25){\circle*{5}}

\put(0,30){\line(0,1){15}}
\put(0,30){\line(1,1){15}}
\put(0,30){\line(1,-1){15}}
\put(15,45){\line(0,-1){30}}


\put(15,45){\line(2,-1){30}}
\put(15,15){\line(2,1){30}}
\put(45,30){\line(0,1){15}}

\put(15,45){\line(-1,2){15}}
\put(15,45){\line(0,1){30}}
\put(15,45){\line(1,1){30}}

\put(0,75){\circle*{5}}
\put(15,75){\circle*{5}}
\put(45,75){\circle*{5}}
\put(23,72){$\cdots$}

\put(15,15){\line(-1,-2){15}}
\put(15,15){\line(0,-1){30}}
\put(15,15){\line(1,-1){30}}

\put(0,-15){\circle*{5}}
\put(15,-15){\circle*{5}}
\put(45,-15){\circle*{5}}
\put(21,-17){$\cdots$}
\put(15,-30){\small $\mathbb{U}^{16}_n$}

\put(90,15){\circle*{5}}
\put(90,45){\circle*{5}}
\put(120,45){\circle*{5}}
\put(120,15){\circle*{5}}
\put(150,45){\circle*{5}}
\put(150,15){\circle*{5}}

\put(90,45){\line(0,-1){30}}
\put(90,15){\line(1,0){30}}
\put(90,45){\line(1,0){30}}
\put(120,45){\line(0,-1){30}}
\put(150,45){\line(0,-1){30}}

\put(132,47){$e_1$}
\put(132,8){$e_2$}
\put(120,45){\line(1,0){30}}
\put(120,15){\line(1,0){30}}

\put(120,45){\line(-1,2){15}}
\put(120,45){\line(0,1){30}}
\put(120,45){\line(1,1){30}}

\put(105,75){\circle*{5}}
\put(120,75){\circle*{5}}
\put(150,75){\circle*{5}}
\put(126,72){$\cdots$}

\put(110,-30){\small $U^{17}_n$}

\put(195,30){\circle*{5}}
\put(210,45){\circle*{5}}
\put(210,15){\circle*{5}}
\put(240,45){\circle*{5}}
\put(240,15){\circle*{5}}

\put(195,30){\line(0,1){15}}
\put(195,45){\circle*{5}}

\put(195,30){\line(1,1){15}}
\put(195,30){\line(1,-1){15}}
\put(210,45){\line(0,-1){30}}
\put(240,45){\line(0,-1){30}}

\put(222,47){$e_1$}
\put(222,8){$e_2$}
\put(210,45){\line(1,0){30}}
\put(210,15){\line(1,0){30}}

\put(210,45){\line(-1,2){15}}
\put(210,45){\line(0,1){30}}
\put(210,45){\line(1,1){30}}

\put(195,75){\circle*{5}}
\put(210,75){\circle*{5}}
\put(240,75){\circle*{5}}
\put(216,72){$\cdots$}
\put(210,-30){\small $U^{18}_n$}

\put(285,30){\circle*{5}}
\put(300,45){\circle*{5}}
\put(300,15){\circle*{5}}
\put(330,30){\circle*{5}}
\put(330,45){\circle*{5}}
\put(285,45){\circle*{5}}

\put(285,30){\line(0,1){15}}
\put(285,30){\line(1,1){15}}
\put(285,30){\line(1,-1){15}}
\put(300,45){\line(0,-1){30}}

\put(312,40){$e_1$}
\put(312,15){$e_2$}

\put(300,45){\line(2,-1){30}}
\put(300,15){\line(2,1){30}}
\put(330,30){\line(0,1){15}}

\put(300,45){\line(-1,2){15}}
\put(300,45){\line(0,1){30}}
\put(300,45){\line(1,1){30}}

\put(285,75){\circle*{5}}
\put(300,75){\circle*{5}}
\put(330,75){\circle*{5}}
\put(306,72){$\cdots$}
\put(300,-30){\small $U^{19}_n$}

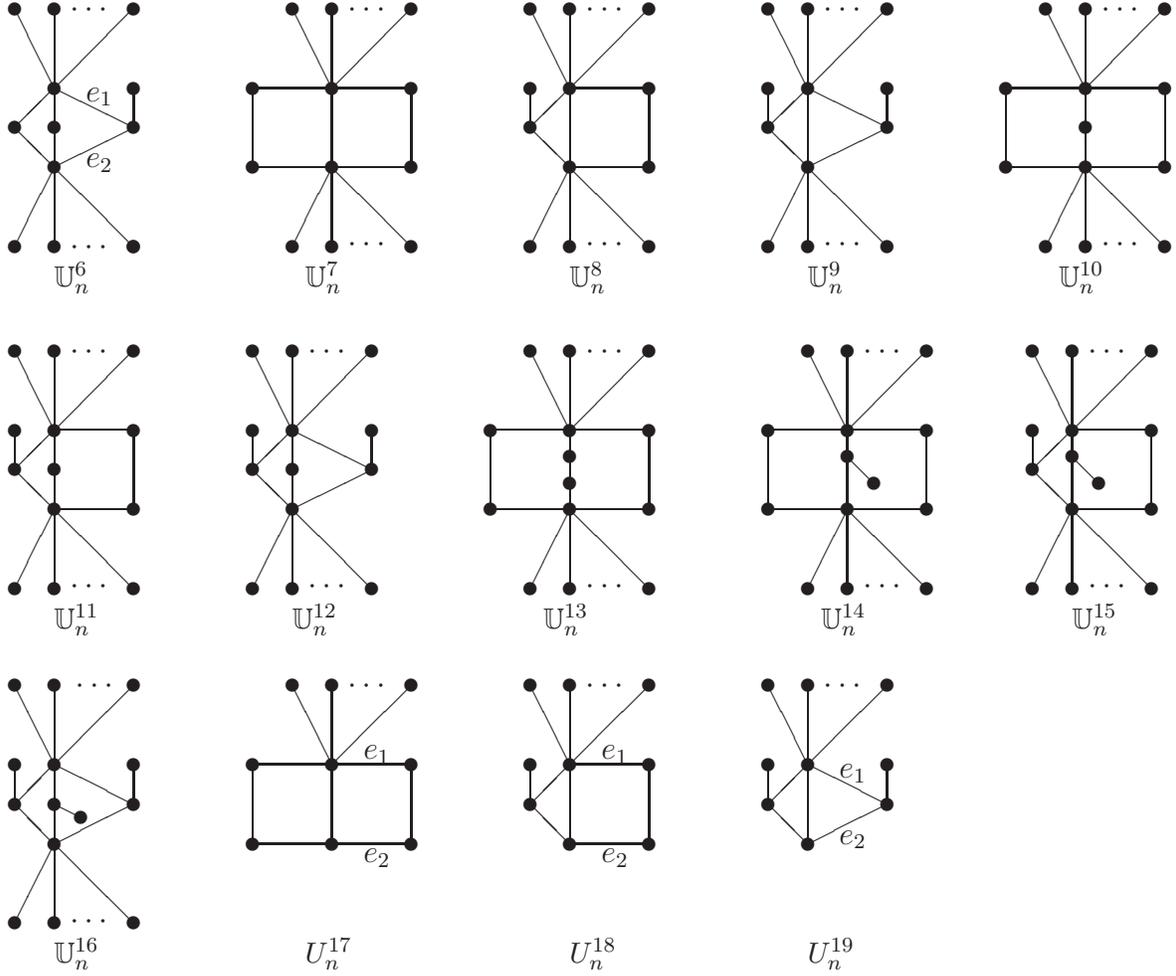
\put(5,-50){\small Fig.3\quad The structures of graphs in  $\mathbb{U}^1_n,$ $\mathbb{U}_n^2,$ \ldots, $\mathbb{U}_n^{16}$ and graphs $U^{17}_n,$ $U^{18}_n,$ $U^{19}_n$.}
\end{picture}

\vspace{2.2cm}
Clearly, the graphs $U^{17}_n,$ $U^{18}_n,$ $U^{19}_n$ are the special graphs in the  $\mathbb{U}^7_n,$ $\mathbb{U}_n^8,$ $\mathbb{U}_n^9,$ respectively.

For $i=0,1,2,3,$ let $T_n^i$ be the tree obtained by attaching $i$ paths with two vertices to the central vertex of $K_{1,n-2i-1},$
where $n\geq2i+1,$ see Fig.4. In particular, $T_n^0=K_{1,n-1}.$
\vskip0.25cm
\begin{picture}(500,100)
\put(65,45){\circle*{5}}

\put(65,45){\line(-1,2){15}}
\put(65,45){\line(0,1){30}}
\put(65,45){\line(1,1){30}}

\put(50,75){\circle*{5}}
\put(65,75){\circle*{5}}
\put(95,75){\circle*{5}}
\put(72,72){$\cdots$}

\put(65,45){\line(-1,-2){15}}
\put(65,45){\line(0,-1){30}}
\put(65,45){\line(1,-1){30}}

\put(50,15){\circle*{5}}
\put(65,15){\circle*{5}}
\put(95,15){\circle*{5}}

\put(50,15){\line(0,-1){20}}
\put(65,15){\line(0,-1){20}}
\put(95,15){\line(0,-1){20}}

\put(50,-5){\circle*{5}}
\put(65,-5){\circle*{5}}
\put(95,-5){\circle*{5}}
\put(72,-7){$\cdots$}

\put(55,88){\scriptsize $n-2i-1$}
\put(70,-22){\scriptsize $i$}
\put(45,77){ $\overbrace{\quad \quad \quad \quad}$}
\put(49,-7){$\underbrace{\quad \quad \quad \quad}$}
\put(77,-30){\small $T^i_n$}
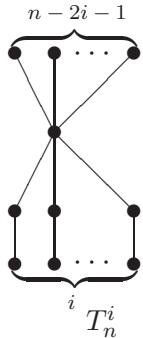
\put(0,-47){\small Fig.4\quad The tree $T^i_n$ with $i=0,1,2,3.$}

\put(250,70){\circle*{5}}
\put(250,70){\line(1,0){20}}
\put(270,70){\circle*{5}}
\put(278,67){$\cdots$}
\put(300,70){\circle*{5}}
\put(300,70){\line(1,0){20}}
\put(320,70){\circle*{5}}

\put(320,70){\line(-1,2){12}}
\put(320,70){\line(0,1){23}}
\put(320,70){\line(1,1){22}}
\put(308,93){\circle*{5}}
\put(320,93){\circle*{5}}
\put(342,93){\circle*{5}}
\put(324,90){$\cdots$}

\put(320,70){\line(1,0){20}}
\put(340,70){\circle*{5}}
\put(348,67){$\cdots$}
\put(370,70){\line(1,0){20}}
\put(370,70){\circle*{5}}
\put(390,70){\circle*{5}}

\put(250,70){\line(0,-1){20}}
\put(250,50){\circle*{5}}
\put(248,30){$\vdots$}
\put(250,20){\circle*{5}}
\put(250,20){\line(0,-1){20}}
\put(250,0){\circle*{5}}
\put(250,0){\line(1,0){20}}
\put(270,0){\circle*{5}}
\put(278,-3){$\cdots$}
\put(300,0){\circle*{5}}
\put(300,0){\line(1,0){20}}
\put(320,0){\circle*{5}}
\put(320,0){\line(1,0){20}}
\put(340,0){\circle*{5}}
\put(348,-3){$\cdots$}
\put(370,0){\circle*{5}}
\put(370,0){\line(1,0){20}}
\put(390,0){\circle*{5}}

\put(390,0){\line(0,1){20}}
\put(390,20){\circle*{5}}
\put(388,30){$\vdots$}
\put(390,50){\circle*{5}}
\put(390,50){\line(0,1){20}}

\put(320,70){\line(0,-1){20}}
\put(320,50){\circle*{5}}
\put(318,30){$\vdots$}
\put(320,20){\circle*{5}}
\put(320,20){\line(0,-1){20}}

\put(320,0){\line(-1,-2){12}}
\put(320,0){\line(0,-1){23}}
\put(320,0){\line(1,-1){22}}
\put(308,-23){\circle*{5}}
\put(320,-23){\circle*{5}}
\put(342,-23){\circle*{5}}
\put(324,-26){$\cdots$}

\put(322,62){$v_1$}
\put(300,61){$v_3$}
\put(340,61){$v_5$}

\put(325,45){$v_7$}
\put(325,19){$v_8$}
\put(322,4){$v_2$}
\put(300,4){$v_4$}
\put(340,4){$v_6$}
\put(280,35){$C_p$}
\put(350,35){$C_q$}
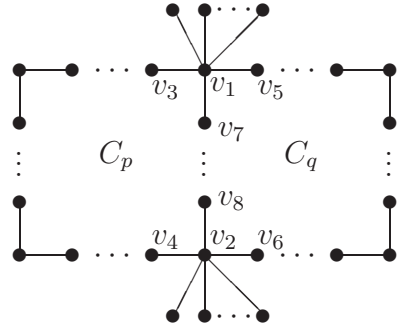
\put(215,-47){\small Fig.5\quad The graph $G$ in the proof of Theorem \ref{thm414}.}
\end{picture}

\vspace{2cm}
From Lemma \ref{lem23} and Lemma 4.6 in \cite{2012b}, we have

\begin{lem}\label{lem413}
{\rm (i)} For $n\geq6,$ we have $1<q_2(T_n^2)<2.7,$ $S_2^+(T_n^2)<e(T^2_n)+2.$
{\rm (ii)}For $n\geq7,$ we have $1<q_2(T_n^3)<2.7,$ $S_2^+(T_n^3)<e(T^3_n)+2.$
\end{lem}

\begin{thm}\label{thm414}
Let $n,k$ be positive integers with $n\geq11$ and $1\leq k\leq n,$ and $G$ be a $\theta$-type bicyclic graph with $n$ vertices. Then
\vskip0.2cm
\hskip.5cm $S^+_k(G)\leq e(G)+\binom{k+1}{2}.$
\end{thm}
\begin{proof}
From Theorem \ref{thm42}, we only need to prove the case of $k=3.$

Since $G$ is a $\theta$-type bicyclic graph, then $G$ is obtained by attaching some hanging trees to $\theta(p,q,t)$ where $p,q\geq3$ and $2\leq t\leq \min\{\frac{p+2}{2},\frac{q+2}{2}\}.$ By Corollary \ref{cor44}, it will suffice to consider
the $\theta$-type bicyclic graph which is obtained by attaching some pendent vertices to $\theta(p,q,t)$ or $G\cong\theta(p,q,t).$

Let $A$ be the set of the common vertices of $C_{p}$ and $C_{q}.$
Let $v_1$, $v_2$ be the two vertices which are the common vertices of the three cycles in $G$.
Let $v_3$, $v_4$ ($v_3$ and $v_4$ may be the same vertex) be the neighbor of $v_1$ and $v_2$ in $V(C_{p})\setminus A,$
$v_5$, $v_6$ ($v_5$ and $v_6$ may be the same vertex) be the neighbor of $v_1$ and $v_2$ in $V(C_{q})\setminus A,$ respectively.
Let $G_1$ be the component of $G-\{v_1v_3, v_2v_4\}$ containing $v_3,$ $G_2$ be the component of $G-\{v_1v_5, v_2v_6\}$ containing $v_5.$
If $|A|\geq3,$ then let $v_7$ and $v_8$ be the neighbor of $v_1$ and $v_2$ in $A,$ respectively ($v_7=v_8$ if $|A|=3$).
Let $G_3$ be the component of $G-\{v_1v_7, v_2v_8\}$ containing $v_7.$ Furthermore, if $|A|=2,$  we define $e(G_3)=0.$

{\bf Case 1:} $e(G_1)\geq2$ or $e(G_2)\geq2$ or $e(G_3)\geq 2.$

Without loss of generality, we suppose $e(G_1)\geq2.$ Then $G$ can be considered as $G_1\approx (G-G_1)$ in which the inserted edges are $v_1v_3$
and $v_2v_4$. Thus by Inequality (1), Theorem \ref{thm41} and Lemma \ref{lem44}, $S^+_3(G)\leq e(G)+6.$

{\bf Case 2:} $e(G_1)\leq1$ and $e(G_2)\leq1$ and $e(G_3)\leq1.$

We show the structure of $G$ in each case in the following.

\vskip0.2cm

\hskip.5cm
\begin{tabular}{|c|c|c|c|}
  \hline
  $e(G_1)$ & $e(G_2)$ & $e(G_3)$ & The structure of $G$       \\ \hline
  $0$      & $0$      &  $0$     &$\mathbb{U}^1_n,$ $\mathbb{U}^2_n$   \\ \hline
  $0$      & $1$      &  $0$     &$\mathbb{U}^3_n,$ $\mathbb{U}^4_n,$ $\mathbb{U}^5_n,$ $\mathbb{U}^6_n$    \\ \hline
  $1$      & $0$      &  $0$     &$\mathbb{U}^3_n,$ $\mathbb{U}^4_n,$ $\mathbb{U}^5_n,$ $\mathbb{U}^6_n$    \\ \hline
  $1$      & $1$      &  $0$     &$\mathbb{U}^7_n,$ $\mathbb{U}^8_n,$ $\mathbb{U}^9_n,$ $\mathbb{U}^{10}_n,$ $\mathbb{U}^{11}_n,$ $\mathbb{U}^{12}_n$    \\ \hline
  $0$      & $0$      &  $1$     &$\mathbb{U}^{6}_n$   \\ \hline
  $0$      & $1$      &  $1$     &$\mathbb{U}^{11}_n,$ $\mathbb{U}^{12}_n$   \\ \hline
  $1$      & $0$      &  $1$     &$\mathbb{U}^{11}_n,$ $\mathbb{U}^{12}_n$   \\ \hline
  $1$      & $1$      &  $1$     &$\mathbb{U}^{13}_n,$ $\mathbb{U}^{14}_n,$ $\mathbb{U}^{15}_n,$ $\mathbb{U}^{16}_n$    \\ \hline
\end{tabular}
\vskip.2cm
\hskip1.5cm
 Table 1. The structure of $G$ in each case.
\vskip0.2cm
Now we show for any $G\in\cup_{i=1}^{16}\mathbb{U}_n^i,$ $S^+_3(G)\leq e(G)+6.$
Without loss of generality, suppose $d_{v_1}\geq d_{v_2}.$ Note that $n\geq11,$ it implies $d_{v_1}\geq5.$

Let $G'=G-\{v_1v_3,v_1v_5,v_1v_7\}$ if $|A|\geq3$ and $G'=G-\{v_1v_3,v_1v_5,v_1v_2\}$ if $|A|=2.$ Let $G_4$ be the component of $G'$ containing $v_1,$ $G_5$ be the component of $G'$ containing $v_2.$ Then $G'=G_4\cup G_5$ and both $G_4, G_5$ are trees. Let $n_i=|V(G_i)|,$ where $i=4,5$ and $n_4+n_5=n.$
Obviously, $G_4\cong T_{n_4}^0$ with $n_4\geq3,$ which implies that $q_1(G_4)=n_4,$ $q_2(G_4)=1.$
\vskip0.1cm
{\bf Subcase 2.1:} $G\in\mathbb{U}_n^1\cup\mathbb{U}_n^2.$

Then $G_5\cong T^0_{n_5},$ then the first three largest singless Laplacian eigenvalues of $G'$ are $n_4,n_5,1,$ that is, $S^+_3(G')=n_4+n_5+1=n+1.$ By Lemma \ref{lem25},

\hskip.5cm $S^+_3(G)\leq S^+_3(G')+3S^+_3(K_2)=n+1+6=e(G)+6.$

{\bf Subcase 2.2:} $G\in\mathop{\cup}\limits_{i=3}^6\mathbb{U}_n^i.$

For each graph $G,$ let $e_1,$ $e_2$ be the edges as labeled in Fig.3. Then the result follows from Lemma \ref{lem410}.

{\bf Subcase 2.3:} $G\in\mathop{\cup}\limits_{i=7}^{12}\mathbb{U}_n^i.$ Then $G_5\cong T^2_{n_5}.$

{\bf Subcase 2.3.1:} $n_5=5.$

Then $G\in \mathop{\cup}\limits_{i=17}^{19}U_n^i.$
For each graph $G,$ let $e_1,$ $e_2$ be the edges as labeled in Fig.3. Then the result follows from Lemma \ref{lem410}.

{\bf Subcase 2.3.2:} $n_5\geq6.$
Then by Lemma \ref{lem413}, we have $1<q_2(G_5)<2.7<3\leq n_4=q_1(G_4),$ which implies that the first three largest signless Laplacian eigenvalues
 of $G'$ are $q_1(G_4)=n_4, q_1(G_5)=n_5, q_2(G_5),$ that is, $S^+_3(G')=n_4+S^+_2(G_5).$ By Lemma \ref{lem25} and (i) of Lemma \ref{lem413}, we have

$S^+_3(G)\leq S^+_3(G')+3S^+_3(K_2)=n_4+S^+_2(G_5)+6<(n_4+e(G_5)+2)+6=e(G)+6.$

{\bf Subcase 2.4:} $G\in\mathop{\cup}\limits_{i=13}^{16}\mathbb{U}_n^i.$

Then $G_5\cong T^3_{n_5}$ with $n_5\geq7.$ By Lemma \ref{lem25} and (ii) of Lemma \ref{lem413}, we can prove the result similar to the proof of subcase 2.3.2.

Combining the above arguments, the result holds.
\end{proof}

By Lemma \ref{lem22}, Theorem \ref{thm42}, Theorem \ref{thm46} and Theorem \ref{thm414}, we have
\begin{thm}\label{thm415}
Conjecture \ref{con11} is true for bicyclic graphs.
\end{thm}

\end{document}